\documentclass[reqno,a4paper]{amsart}

\usepackage{latexsym,amssymb}

\usepackage[ansinew]{inputenc}
\usepackage[T1]{fontenc}

\usepackage{times,mathptmx}

\usepackage{color}
\usepackage{graphicx,overpic}
\usepackage{tikz}

\usepackage{hyperref}

\theoremstyle{plain}

\newtheorem{prop}{Proposition}[section]
\newtheorem{Thm}[prop]{Theorem}
\newtheorem{lemma}[prop]{Lemma}

\theoremstyle{definition}

\theoremstyle{remark}
\newtheorem{rem}[prop]{Remark}

\newcommand{\IN}{\mathbb{N}}

\newcommand{\IR}{\mathbb{R}}

\newcommand{\M}{{{\mathcal M}}}

\newcommand{\C}{\mathcal{C}}
\newcommand{\abs}[1]{\mathopen\vert#1\mathclose\vert}

\newcommand{\norm}[1]{\mathopen\Vert#1\mathclose\Vert}
\newcommand{\intd}{\,{\mathrm d}}
\renewcommand{\phi}{\varphi}
\renewcommand{\epsilon}{\varepsilon}
\renewcommand{\le}{\leqslant}

\renewcommand{\subset}{\subseteq}

\definecolor{umhblue}{rgb}{.69,.75,.86}
\definecolor{umhdarkblue}{rgb}{.15,.15,.53}% for ``Darmstadt'' theme
\definecolor{umhred}{rgb}{0.4,.0,.0}
\definecolor{mygreen}{RGB}{0, 165, 0}
\definecolor{definition}{rgb}{0.45, 0.61, 0.96}
\definecolor{problem}{rgb}{0.84,0.5,0}

\providecommand{\meshline}[5]{%
  \begin{pgfscope}
    \pgfsetcolor{#1}
    \pgfpathmoveto{\pgfpointxy{#2}{#3}}
    \pgfpathlineto{\pgfpointxy{#4}{#5}}
    \pgfusepath{stroke}
  \end{pgfscope}}

\begin{document}

\title[Nonlinear Schrödinger problems]{Nonlinear Schrödinger problems:
  symmetries of some variational solutions} 

\author[Ch.~Grumiau]{Christopher Grumiau}

\address{
  Institut de Math{\'e}matique\\
  Universit{\'e} de Mons, Le Pentagone\\
  20, Place du Parc, B-7000 Mons, Belgium}
\email{Christopher.Grumiau@umons.ac.be}

\thanks{The author is partially supported by a grant from the
  National Bank of Belgium and  by the program 
``Qualitative study of
 solutions of variational elliptic partial differential equations. Symmetries,
bifurcations, singularities, multiplicity and numerics'' of the FNRS, project 
2.4.550.10.F of the Fonds de la Recherche Fondamentale Collective. }

\begin{abstract}
In this paper, we are interested in the nonlinear Schrödinger problem 
$-\Delta u + Vu = \abs{u}^{p-2}u$ submitted to the
 Dirichlet boundary conditions. We consider  $p>2$ and we are working
 with an open bounded domain $\Omega\subset\IR^N$ ($N\geq 2$).
 Potential $V$  satisfies $\max(V,0)\in L^{N/2}(\Omega)$ and $\min(V,0)\in
 L^{+\infty}(\Omega)$. Moreover, $-\Delta + V$ is  
positive definite and has one and only one principal eigenvalue. 
When $p\simeq 2$, we prove the uniqueness of the solution once we
fix the projection on an eigenspace of $-\Delta + V$. It implies partial 
symmetries (or symmetry breaking) for ground state and least 
energy nodal solutions. In the litterature, the case $V\equiv 0$ has
already been studied. Here, we generalize the technique at our case by pointing
out and explaining differences. To finish, as
illustration, we   implement the (modified) mountain pass algorithm 
to work with $V$ negative, piecewise constant or not bounded. It
permits us to exhibit direct examples where the
solutions break down  the symmetries of $V$.
\end{abstract}

%% Four or five keywords or phrases

\maketitle

\section{Introduction}
Let $N\geq 2$, $p>2$, $\lambda>0$ and an open bounded domain 
$\Omega\subset\IR^N $. We  study the nonlinear Schrödinger problem 
\begin{equation}
\label{eqp}
-\Delta u(x) + V(x)u(x) =  \lambda|u(x)|^{p-2}u(x)
\end{equation}
submitted to the Dirichlet boundary conditions (DBC).  We are
interested in the symmetry of solutions. 

When $V$ 
belongs to $L^{N/2}(\Omega)$, the solutions can be defined as  the critical 
points of the energy functional 
\begin{equation*}
%\label{eqp}
\mathcal{E}_p: H^1_0
(\Omega)\to\IR : u\mapsto \frac{1}{2}\int_{\Omega}\abs{\nabla u}^2 + Vu^2 
-\frac{\lambda}{p}\int_{\Omega}\abs{u}^p.
\end{equation*}

Clearly, $0$ is solution. Concerning other solutions, if we assume
that $-\Delta + V$ is 
positive definite and $V^-:=\min(V,0)\in L^{+\infty}(\Omega)$, then the norm 
$\norm{u}^2=\int_{\Omega}\abs{\nabla u}^2 + Vu^2$ defined on 
$H^1_0(\Omega)$ is equivalent to the traditional norm $\norm{u}^2_{H^1_0}=
\int_{\Omega}\abs{\nabla u}^2$ (see Proposition~\ref{equiv}). By
working in the same way as in~\cite{neuberger}, it directly implies the 
existence of ground state solutions (g.s.) and least energy 
nodal solutions (l.e.n.s.); i.e.\ one-signed (resp.\ sign-changing) solutions 
with minimal energy. These solutions are 
characterized as minima of $\mathcal{E}_p$ respectively on  the
(resp.\ nodal) Nehari set
\[\mathcal{N}_p:=\left\{u\in H^1_0(\Omega)\setminus\{0\}\bigm\vert 
\int_{\Omega}\abs{\nabla u}^2 + Vu^2 =\lambda\int_{\Omega}\abs{u}^p\right\}\]
(resp.\ $\M_p:=\{u : u^\pm\in \mathcal{N}_p\}$). The Morse index is $1$
(resp.\ $2$).

In this paper, we study the structure of these two types of
solutions.  We verify  whether they
are odd or even with respect to the hyperplanes leaving $V$ invariant (i.e.\ 
$V$ respects an orthogonal symmetry with respect to the hyperplane).
When it is the case, we say that the solution respects the symmetries
of $V$. When $V\equiv 0$, this type of questions has already been
studied. First, 
on the square in 
dimension $2$, we can mention a result of 
G.~Arioli and H.~Koch (see~\cite{arioli}). They proved the existence of a
positive symmetric 
$\C^{\infty}$-function $w$ such that $-\Delta u = wu^3$ possesses
a non-symmetric positive solution (with $1$ as Morse index). The same
kind of result has also been obtained for a solution with $2$ as Morse
index. The proof is partially computer-assisted. Second, in collaboration 
with D.~Bonheure, V.~Bouchez, C.~Troestler and J.~Van~Schaftingen 
(see~\cite{bbg,bbgv,gt}), we proved for $p$ close to $2$ that the
symmetries are related to 
the symmetries of eigenfunctions of $-\Delta$. We generalize here the
technique at some non-zero potentials $V$  and we make 
numerical  experiments to illustrate it.

In
Section~\ref{sec:abst-sym}, by denoting 
$\lambda_{i}$ (resp.\ $E_i$) the distinct eigenvalues (resp.\ eigenspaces) of 
$-\Delta+ V$ with DBC in $H^1_0(\Omega)$, we prove the following
Theorem~\ref{intro}. For 
this, we assume that  $\lambda_1$ is the unique principal  
eigenvalue, i.e.\ an eigenvalue with a related eigenspace of dimension $1$ possessing an
one-signed eigenfunction. We also require  
that  eigenfunctions in $E_2$ have a nodal line of measure $0$. By
using a maximum principle, these
assumptions are  
satisfied at least when $V\in L^{+\infty}(\Omega)$ (see~\cite{gossez}).
\begin{Thm}
\label{intro}
When $V\in L^{N/2}(\Omega)$, $V^-\in L^{+\infty}(\Omega)$ and $-\Delta
+ V$ is positive definite such that $\lambda_1$ is the unique
principal eigenvalue, for $p$ 
close to $2$,  the ground state\  (resp.\ least energy nodal) solutions   
respect the symmetries of their orthogonal projections in
$H^1_0(\Omega)$ on $E_1$ (resp.\ $E_2$).
\end{Thm}

In particular, when the eigenspace has a dimension $1$, the solutions
respect the symmetries of $V$. As we assumed that $\lambda_1$ is the
unique principal eigenvalue, ground state solutions respect the
symmetries of $V$. 

Depending on the structure of $E_2$, 
some symmetry breaking exist for  
least energy nodal solutions (see Section~\ref{symbr}). 
In fact, by a traditional bootstrap, a family of ground state (resp.\
least energy nodal) 
solutions $(u_p)_{p>2}$  converges for 
$\C$-norm to functions in $E_1$ (resp.\ $E_2$). So, for
l.e.n.s., $u_p$ does not    
respect the symmetries of $V$ for $p$ small when  the projection is not
symmetric in $E_2$ (see Section~\ref{exp} for an example). For larger $p$, it
is depending on the case. In
Section~\ref{symbr},  we exhibit rectangles and $V$ (such that the
eigenfunctions in $E_2$ are symmetric) where   
l.e.n.s.\ do not respect symmetries of $V$ for $p$ large enough.  So,
the result~\ref{intro} cannot be extended to all $p$.

In Section~\ref{exp}, as illustration, we implement the (modified) mountain
pass algorithm (see~\cite{mckenna,zhou1,zhou2}) to  
study the cases of $V$ negative constant, piecewise constant or
singular. We exhibit direct examples such that the 
solutions break down  the symmetries of $V$. 
% For alignments use AmS-LaTeX constructions not \eqnarray.

\section{Main results}\label{sec:abst-sym}
%\vspace{-0.3cm}

The proofs are related to the technique defined in~\cite{bbgv}. This is why we 
just point out and explain the differences and we do not make all the details.
The first result implies that the traditional Poincaré's and 
Sobolev's inequalities are available for 
$\norm{\cdot}^2:=\int_{\Omega}\abs{\nabla \cdot}^2 + V(\cdot)^2$. 
\begin{prop}
\label{equiv}
If $-\Delta +V$ is positive definite, $V^+\in L^{N/2}(\Omega)$ and 
$V^-\in L^{+\infty}(\Omega)$, the norm  $\norm{u}^2:=\int_{\Omega}\abs{\nabla u}^2 + 
Vu^2$ and the traditional norm 
$\norm{u}_{H^1_0}^2:=\int_{\Omega}\abs{\nabla u}^2$ are equivalent. 
\end{prop}

\begin{proof}
Using the Sobolev's inequalities on $\int_{\Omega}Vu^2$ and as $V\in
L^{N/2}(\Omega)$,  $\exists C>0$ such that 
\begin{equation*}
\norm{u}^2 \leq  \int_{\Omega} \abs{\nabla u}^2 + C \int_{\Omega}
\abs{\nabla u}^2.
\end{equation*}

Using the Poincaré's inequalities and as 
$V^-\in L^{+\infty}(\Omega)$, $\exists C>0$ and a real $K$ such that
\begin{equation*}
\begin{split}
\norm{u}^2 &= \varepsilon \int_{\Omega}\abs{\nabla u}^2 + (1-\varepsilon) 
\norm{u}^2  + \varepsilon \int_{\Omega} Vu^2\\
& \geq  \varepsilon\int_{\Omega} \abs{\nabla u}^2
 + ((1 -\varepsilon)C + \varepsilon K)\int_{\Omega}u^2\geq \varepsilon 
\int_{\Omega} 
\abs{\nabla u}^2,
\end{split}
\end{equation*}
where the last inequality is obtained for $\varepsilon$ small enough. 
\end{proof}

Then, the proof of Theorem~\ref{intro} is based on two main results. 
The first one shows that, for $p\simeq 2$, a priori bounded solutions
can be distinguished by their projections on $E_i$. 
\subsection{Abstract symmetry}
\begin{lemma}\label{Lem:uniq}
There exists $\varepsilon >0$ such that  if $\|a(x)-\lambda_{i}\|_{L^{N/2}}<
\varepsilon$ and $u$ solves $-\Delta u + Vu= a(x)u$ with DBC then $u=0$ or 
$P_{E_{i}}u \neq 0$.
\end{lemma}

\begin{proof}
Similar as in Lemma~$3.1$ in~\cite{bbgv}, Poincar\'e's and Sobolev's inequalities 
are adapted using Proposition~\ref{equiv}. 
\end{proof}

Then, we directly obtain our abstract symmetry result as in the proof of 
Proposition~$3.2$ in~\cite{bbgv}. Let us remark that the result holds
for any $i$ and not just for $i=1$ or $2$ as stated in~\cite{bbgv}.  
We denote by $B(0,M)$ the ball in $H^1_0(\Omega)$ centered at $0$ and
radius $M$.

\begin{prop}\label{Prop-intro:uniqueness}
Let $M > 0$.  For $i\in\IN_0,
\exists\tilde{p} > 2$ such that, for 
$p \in (2, \tilde{p})$, if $u_p, v_p\in \{ u\in B(0,M): P_{E_i}u\notin
B(0, \frac{1}{M})\}$ solve the boundary value
problem with DBC $-\Delta u + Vu = \lambda_i \abs{u}^{p-2}u$ then
$P_{E_i}u_p = P_{E_i}v_p$ implies $u_p=v_p$. 
\end{prop}

These two results permit us to conclude as in Theorem~$3.6$ in~\cite{bbgv}. 

\begin{Thm}
\label{thmf}
Let $(G_\alpha)_{\alpha \in E}$ with $E=E_i$ be a group acting on
$H^{1}_{0}(\Omega)$ such that, for $g\in G_{\alpha}$ and $u\in H^{1}_{0}(\Omega)$,
\begin{center}
$g(E)=E$,\quad $g(E^{\perp})=E^{\perp}$,\quad $g\alpha=\alpha$\quad and\quad 
$\mathcal{E}_{p}(gu)=\mathcal{E}_{p}(u).$
\end{center}
For any $M>1$, $\exists \tilde{p}>2$ such that, for any family of solutions 
$(u_p)_{\tilde{p}>p>2}\subset  \{ u\in B(0,M): P_{E_i}u\notin
B(0, \frac{1}{M})\}$ of the boundary value problem with DBC  $-\Delta
u + Vu = \lambda_i \abs{u}^{p-2}u$,  $u_{p}$ belongs to
the invariant set of $G_{\alpha_{p}}$ where $\alpha_{p}$ is the
orthogonal projection $P_{E}u_{p}$.
\end{Thm}

Theorem~\ref{thmf} can be used for any bounded family of
solutions staying away 
 from  $0$. To apply Theorem~\ref{thmf} at a family $(u_p)_{p>2}$ of ground state
 (resp.\ least energy nodal) solutions for the problem~\eqref{eqp}, we
 study the asymptotic 
 behavior  when $p\to 2$. We 
 prove that the expected upper and lower bounds are fine if and only if
 $\lambda = \lambda_1$ (resp.\ $\lambda_2$).  In some sense,
 $\lambda_1$ (resp.\ $\lambda_2$) is the natural rescaling to work
 with ground state (resp.\ least energy nodal) solutions of
 problem~\eqref{eqp}. Let us 
 remark that this condition is not a restriction. Indeed, by
 homogeneity of $\lambda \abs{u}^{p-2}u$ in the equation~\eqref{eqp}, the symmetries of ground state (resp.\
 least energy nodal) solutions are independent of $\lambda$.  

\subsection{Asymptotic behavior}
Let us denote $(u_p)_{p>2}$ a family of ground state
 (resp.\ least energy nodal) solutions for the problem~\eqref{eqp}. 
We consider $\lambda = \lambda_1$ for g.s.\ (resp.\ $\lambda_n$ the first not 
principal eigenvalue for l.e.n.s.) and $E=E_1$ (resp.\ $E_n$). 

\begin{lemma}
Concerning the upper bound, $\limsup_{p\to2} \norm{u_p}^2 = 
\limsup_{p\to 2} \left( \frac{\mathcal{E}_p(u_p)}
{1/2 -1/p}\right)\leq \norm{u_*}^2$ where $u_*\in E$ mi\-ni\-mi\-zes the limit 
functional $\mathcal{E}_*:E\to \IR: u\mapsto \int_{\Omega}u^2- \log u^2.$
\end{lemma}

\begin{proof}
The proof is inspired by Lemma~$4.1$ in~\cite{bbgv}.    First, we define 
$v_p:= u_* + (p-2)w$ where $w\in H^1_0(\Omega)$ solves the problem 
$-\Delta w + V w -\lambda_2 w = \lambda_2 u_* \log \abs{u_*}$ with $P_{E}w=0$.
Then, we prove that the projection of $v_p$ on $\mathcal{N}_p$ (resp.\
$\mathcal{M}_p$) converges when $p\to 2$.

Concerning least energy nodal solutions, in~\cite{bbgv}, the result
has been stated for 
$n=2$. Here, let us remark that it works with $E_n$ which is
not specially $E_2$. We just need to ensure that $v_p$ is
sign-changing for $p$ close to $2$. 
\end{proof}

Nevertheless, we need to assume
$n=2$ to obtain 
the lower bound. It is explained in the next Lemma. 

\begin{lemma}
Concerning the lower bound, if $n=2$ then $\liminf_{p\to 2}\norm{u_p}>0$.
\end{lemma}

\begin{proof}
The proof is inspired by Lemma $4.4$ in~\cite{bbgv}. Concerning l.e.n.s.\ (the
 argument is easier for g.s.), let $e_1$ be a first 
eigenfunction in $E_1$. By considering $s_{p}^-:= 
\frac{\int_{\Omega}u_{p}^+e_1}{\int_{\Omega}\abs{u_p}e_1} $ and
$s_{p}^+:= 1-s_{p}^-$,  we show the existence 
of $t_p>0$ such that $v_p= t_p(s_p^+ u_p^+ + s_p^- u_p^-)$ belongs to 
$\M_p \cap E_1^{\perp}$.

Then, we prove that $v_p$ stays away from zero using Poincaré's and
 Sobolev's embeddings, which concludes the proof.  For this part, we need 
to require that $\lambda_1$ is the unique principal eigenvalue, i.e.\
$n=2$.  Otherwise,  
we should prove that $v_p\in (E_1 \oplus\ldots E_{n-1})^{\perp}$, which
cannot be assumed. 
\end{proof}

The two previous results imply Theorem~\ref{result}.

\begin{Thm}\label{result}
Assume that $-\Delta + V$ is positive definite and possesses one and only 
one principal eigenvalue ($n=2$), $V^+\in L^{N/2}(\Omega)$ and $V^-\in 
L^{+\infty}(\Omega)$. If $(u_{p})_{p>2}$ is a family of ground state\
(resp.\ least energy nodal) solutions for equation~\eqref{eqp} then
 $\exists C>0$ such that $\norm{u_p}_{H^1_0} \le C
 \left(\frac{\lambda_i}{\lambda} 
\right)^\frac{1}{p-2},$ 
where $i=1$ (resp.\ $2$).
If $p_{n}\to 2$ and
$\left(\frac{\lambda_{i}}{\lambda}\right)^{\frac{1}{2-p_{n}}}u_{p_{n}}
\rightharpoonup u_{*}$ in $H^{1}_{0}(\Omega)$, then
$\left(\frac{\lambda_{i}}{\lambda}\right)^{\frac{1}{2-p_{n}}}u_{p_{n}}\to u_{*}$ in
$H^{1}_{0}(\Omega)$, $u_{*}$ satisfies $-\Delta u_{*}+ Vu_*=\lambda_{i}u_{*}$ and
$\mathcal{E}_{*}(u_{*})=\inf \{\mathcal{E}_{*}(u) : u\in
E_{i}\setminus\{0\}, \langle \intd \mathcal{E}_{*}(u),u\rangle=0\},$ where
$\mathcal{E}_{*}:E_{i}\to \IR : u\mapsto\frac{\lambda_{i}}{2}
\int_{\Omega}u^{2}-u^{2}\log u^{2}.$
\end{Thm}

\begin{rem}
\begin{enumerate}[(i)]
\item By a traditional bootstrap,
  $\left(\frac{\lambda_{i}}{\lambda}\right)^{\frac{1}{2-p_{n}}}u_{p_{n}} 
\rightharpoonup u_{*}$  implies
$\left(\frac{\lambda_{i}}{\lambda}\right)^{\frac{1}{2-p_{n}}}u_{p_{n}} \to u_*$ for 
$\mathcal{C}$-norm (see~\cite{gt}).
\item If $\lambda <\lambda_1$ (resp.\ $\lambda_2$), a family of 
g.s.\  (resp.\ l.e.n.s.)  blows up in $H^1_0(\Omega)$. 
If  $\lambda >\lambda_1$ (resp.\ $\lambda_2$), it goes to $0$. So, a family of 
g.s.\ (resp.\ l.e.n.s.) is bounded and stays away from $0$ 
if and only if $\lambda =\lambda_1$ (resp.\ $\lambda_2$).
\item By homogeneity of $\lambda\abs{u}^{p-2}u$,  the study of 
symmetries for only one value of $\lambda$ is enough to conclude
symmetries for any $\lambda >0$.  
\item By combining Theorem~\ref{thmf} and Theorem~\ref{result}, we
  obtain that gound state solutions for $p$ close to
  $2$ respect the symmetries of their projection on $E_1$.  As first
  eigenfunctions are unique up to a constant, they   keep symmetries
  of $V$ for $p$ close to $2$.  
\item By combining Theorem~\ref{thmf} and Theorem~\ref{result}, we
  obtain that least energy nodal solutions for $p$ close to
  $2$ respect the symmetries of their projection on $E_2$. 
\end{enumerate}
\end{rem}
%\vspace{-0.6cm}

\subsection{Symmetry breaking for least energy nodal
  solutions}\label{symbr} 
For $p\simeq 2$, previous results showed that the structure of least
energy nodal solutions are 
related to the symmetries of $u_*$ verifying $\mathcal{E}_{*}(u_{*})=
\inf \{\mathcal{E}_{*}(u) : u\in
E_{i}\setminus\{0\}, \langle \intd \mathcal{E}_{*}(u),u\rangle=0\}.$  

In~\cite{bbgv} (see Section~$6$), on the square and for $V\equiv 0$,
it is proved  that if $u_*$ 
does not respect the symmetries of the rectangle, i.e.\ $u_*$ is not
odd or even  
with resepct to the medians (which is numerically observed), 
then there exists a symmetry breaking on rectangles sufficiently close
to the square and $p$ sufficiently large.  In our case, this property can be
stated as follows. 
\begin{Thm}
Let us work on a square. If $V$ is odd or even with respect to a
median but $u_*$ does not 
respect this symmetry,  then there exist some rectangles and $p$ such that least
energy nodal solutions $u_p$ does not respect the symmetries of $V$.
\end{Thm}

Moreover, as $u_p$ converges for $\C$-norm, we are able to directly
construct $V$ such that the least energy nodal solutions break down
the symmetry of $V$. It will happen once $u_*$ is not symmetric.
In the next section numerical experiments illustrate this interesting 
case.

\section{Numerical illustrations: non-zero potentials $V$}    
\label{exp}
   %%%%%%%%%%
%%%%%%%%%%%%%%%%%%%%%%%%%%%%%%%%%%%%%%%%%%%%%%%%%%%%%%%%%%%%%%%%%%%%%%%%%%%%
In this section, we compute the (resp.\ modified) mountain pass algorithm 
to approach one-signed (resp.\ sign-changing) 
solutions (see~\cite{mckenna,zhou1,zhou2,cdn}).  While it 
is not sure that approximate solutions have least energy, all the other 
solutions that we have found numerically have a larger energy. So,  we will 
assume that the approximations are ground state  
 (resp.\ least energy nodal) solutions.  We also give some level
 curves: $1$ and  
$2$ for g.s., $\pm 1$ and $\pm 2$ for l.e.n.s. 
%\vspace{-0.6cm}

Numerically, we study $p=4$. Let us remark that we always obtain the
same kind of symmetry for smaller values of $p$. We work with
$p=4$ to illustrate that the result of Theorem~\ref{intro} seems to
hold  at least for a non-negligeable interval.

\subsection{Negative constant potential on a square}

As first example, we consider a constant $V$ such that $\lambda_1>0$,
i.e.\ $V> -\tilde{\lambda}_1$ where $\tilde{\lambda}_1$ is the first
eigenvalue of $-\Delta$. So, the required assumptions on $V$ are  
clearly satisfied. Theorem~\ref{intro} holds. In particular,
concerning symmetries, we obtain 
\begin{enumerate}
\item for $p$ close to $2$, on convex domains, ground state solutions
  are even with  respect to each hyperplane leaving $\Omega$ invariant; 
\item for $p$ close to $2$, least energy nodal solutions on a rectangle are
even and odd with respect to a median;
\item for $p$ close to $2$, least energy nodal solutions  on radial domains 
are even with respect to $N-1$ orthogonal directions and odd with respect to 
the orthogonal one; 
\item for $p$ close to $2$, least energy nodal solutions  on a square are odd 
with respect to the barycenter. 
\end{enumerate}
%\vspace{0.3cm}

Numerically, we consider $-\Delta u
- \frac{\pi^2}{4}u = u^3$ 
defined on the square $\Omega=(-1,1)^2$ in $\IR^2$.  First and second
eigenvalues 
 of $-\Delta$ are given by $\frac{\pi^2}{2}$ and
 $\frac{5\pi^2}{4}$. On the following 
graph, one-signed (resp.\ nodal)
numerical solutions have the expected symmetries. Ground state
solutions respect the symmetries of
the square and least energy nodal solutions are odd with respect to
the center $0$. Moreover,  
the nodal line of the 
sign-changing solutions seems to be a diagonal, as for $V\equiv 0$ 
(see~\cite{bbgv}). 

  %\raisebox{-20ex}{%
\begin{minipage}[h]{0.6\linewidth}
   % \begin{figure}
     \includegraphics[width=2.8cm, angle=270]{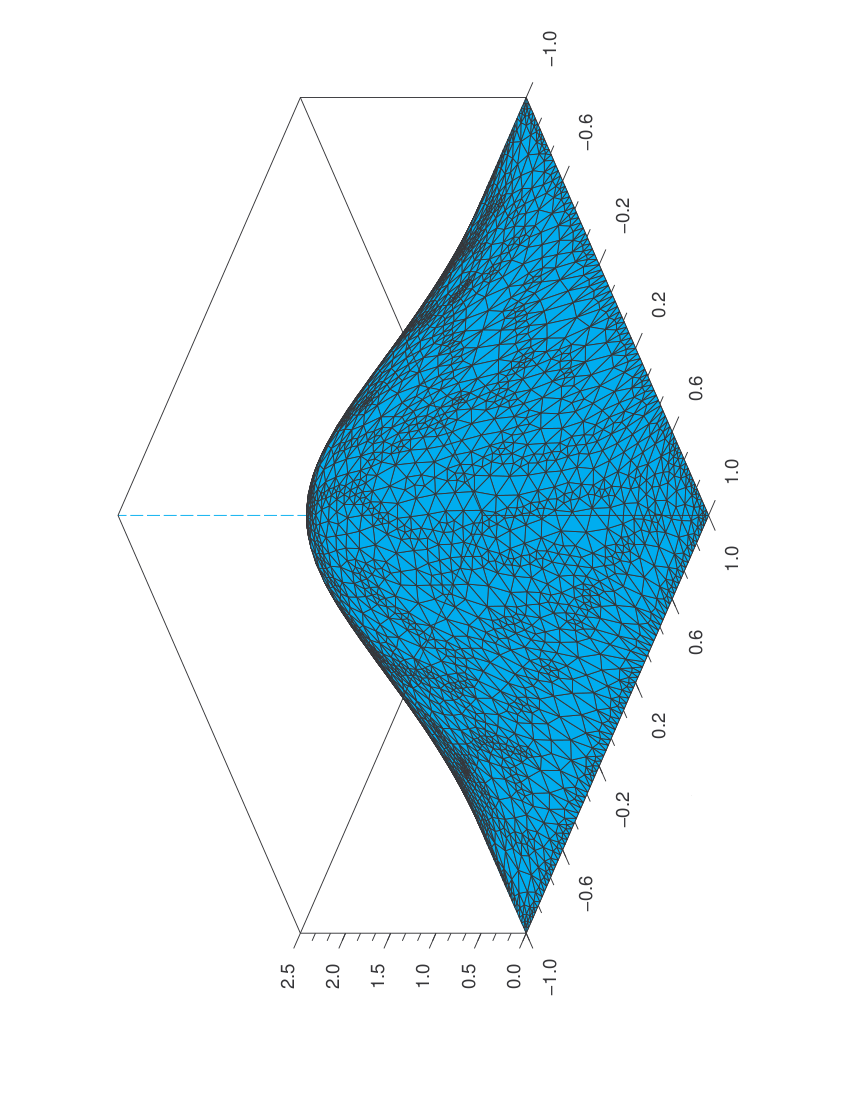}
\hfill
\includegraphics[width=2.8cm, angle = 270]{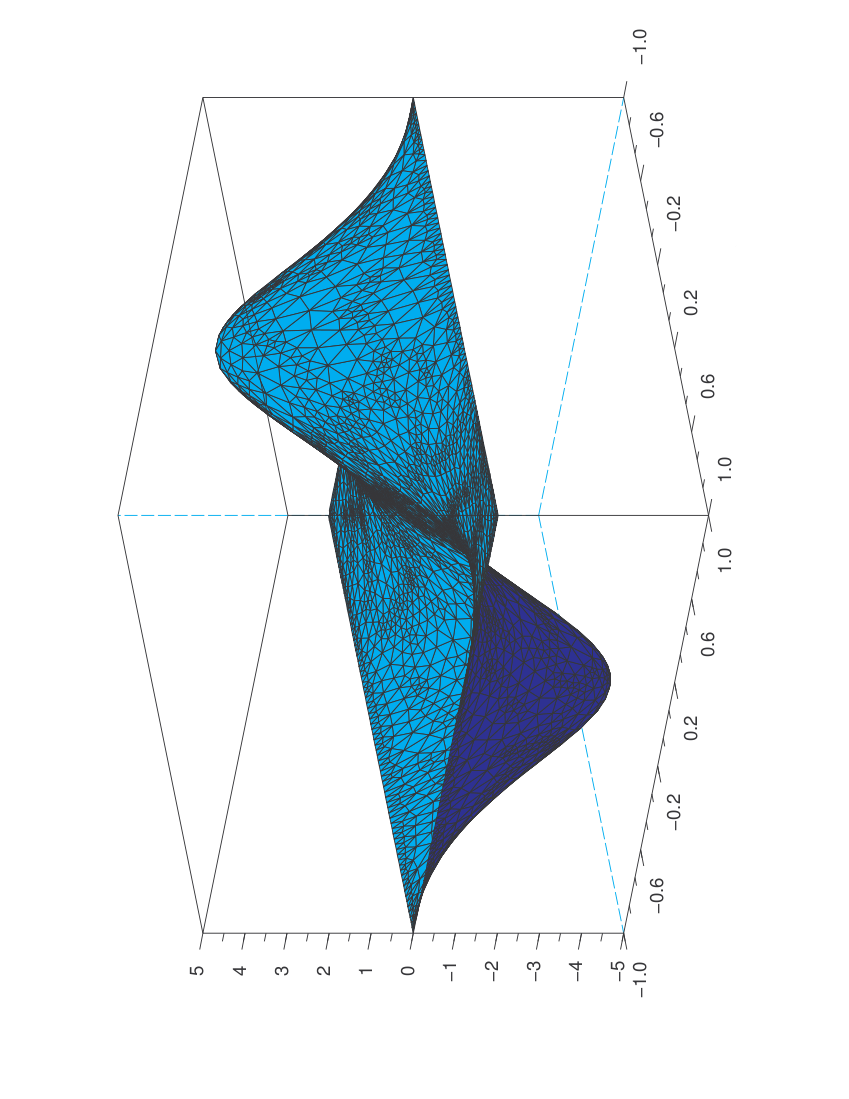}
\hfill\null
\begin{center}

\hfill\null
\end{center}
    \end{minipage}
\hfill
\begin{minipage}[h]{0.4\linewidth}
\begin{itemize}
\item For g.s.: $\max (u)=2.18 $, $\mathcal{E}_{4} (u)= 2.54$
\item For l.e.n.s.: $\min(u)= -4.61$, $\max (u)= 4.61$,  
$\mathcal{E}_4(u)=33.21$
\item Starting function for g.s.:\\ $(x-1)(y-1)(x+1)(y+1)$
\item Starting function for l.e.n.s.:\\ $\sin(\pi(x+1))\sin(2\pi (y+1))$
\end{itemize}
\end{minipage}
\hfill\null
%\begin{table}[!ht]
%\begin{tabular}{c c r@{.}l r@{.}l r@{.}l}
%& Initial function & \multicolumn{2}{c}{$\min u$} & \multicolumn{2}{c}{$\max u$%}
% & \multicolumn{2}{c}{$\mathcal{E}_4(u)$}\\
%\hline $g.s.$ & $(x-1)(y-1)(x+1)(y+1)$ & 0&0 &  2&31 & 2&99 \\
%\hline $l.e.n.$ & $\sin(\pi(x+1))\sin(2\pi (y+1))$ & $-$4&61 & 4&61 & 33&21\\
%%\hline $\lambda_2(N)$ & $\cos(\pi (y+1))$ &-1&31 & 1&33 & 0&78 & 0&77\\
%\hline
%\
%end{tabular}
%\caption{Characteristics of the approximate solutions for negative constant 
%$V$.}
%\label{tableV}
%\end{table}
%\vspace{-0.5cm}

\subsection{Piecewise constant potential on a rectangle}

As second example, $V$ is piecewise constant on $(0,2)\times(0,1)$.
In~\cite{gossez}, it is  
proved that there exists just one principal eigenvalue. So, our assumptions are 
satisfied and Theorem~\ref{intro} is available. For $\lambda=1$, $p=4$ and 
$V(x,y)=V_-:=0$ when $x<1$ (resp.\ $V_+:= 10$ otherwise), following
graphs indicate  
that  approximations are just even with respect to a direction. Ground
state solutions  are more or  less ``located'' in $x<1$  
(the side minimizing energy) and respect symmetries of $V$. Least
energy nodal solutions seem 
to be formed by g.s.\ on each nodal domains and are even with respect
to  a direction.  \\
 %\raisebox{-20ex}{%
\begin{minipage}[h]{0.6\linewidth}
   % \begin{figure}
     \includegraphics[width=2.8cm, angle=270]{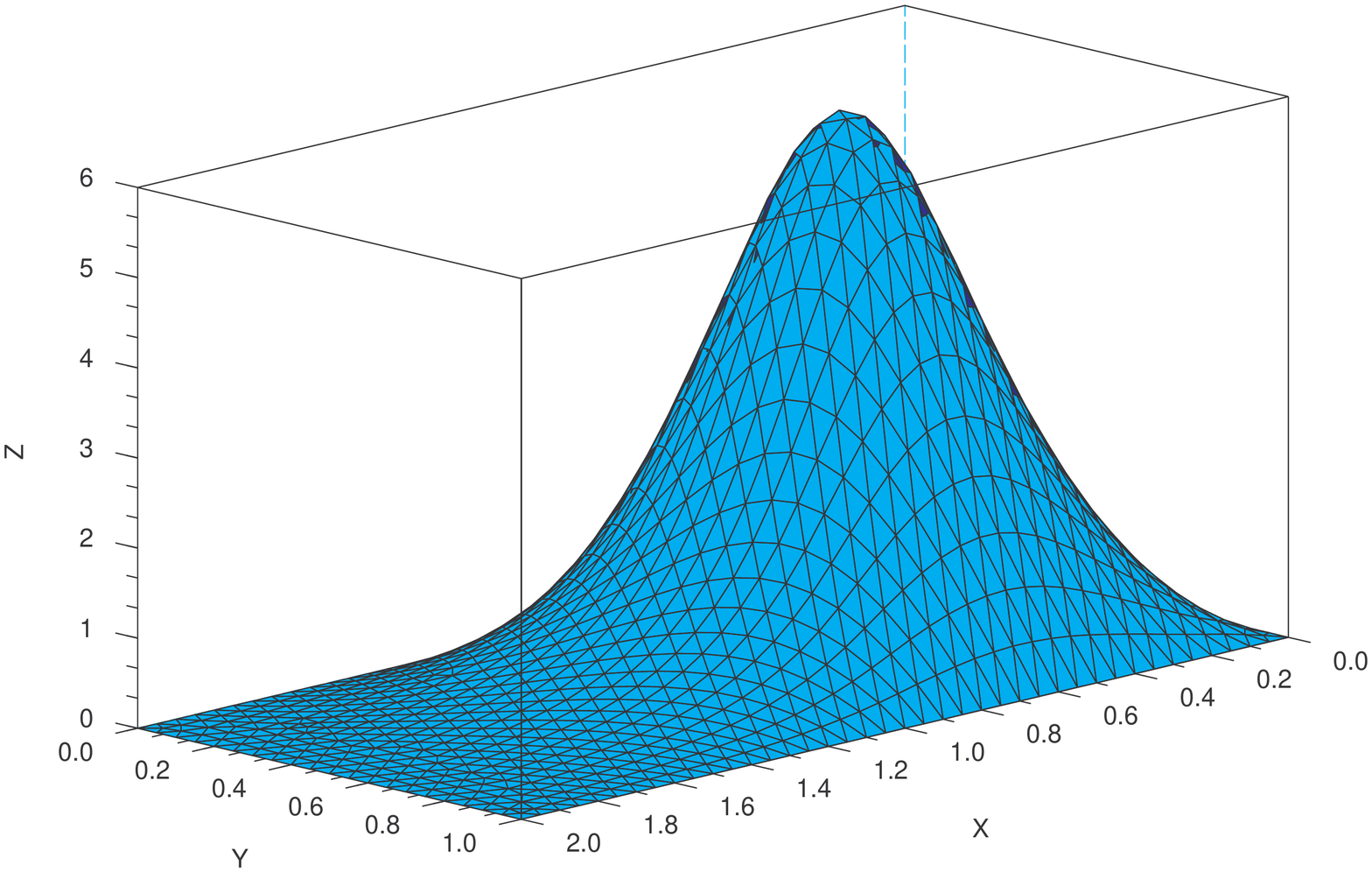}
\hfill
\includegraphics[width=2.8cm, angle = 270]{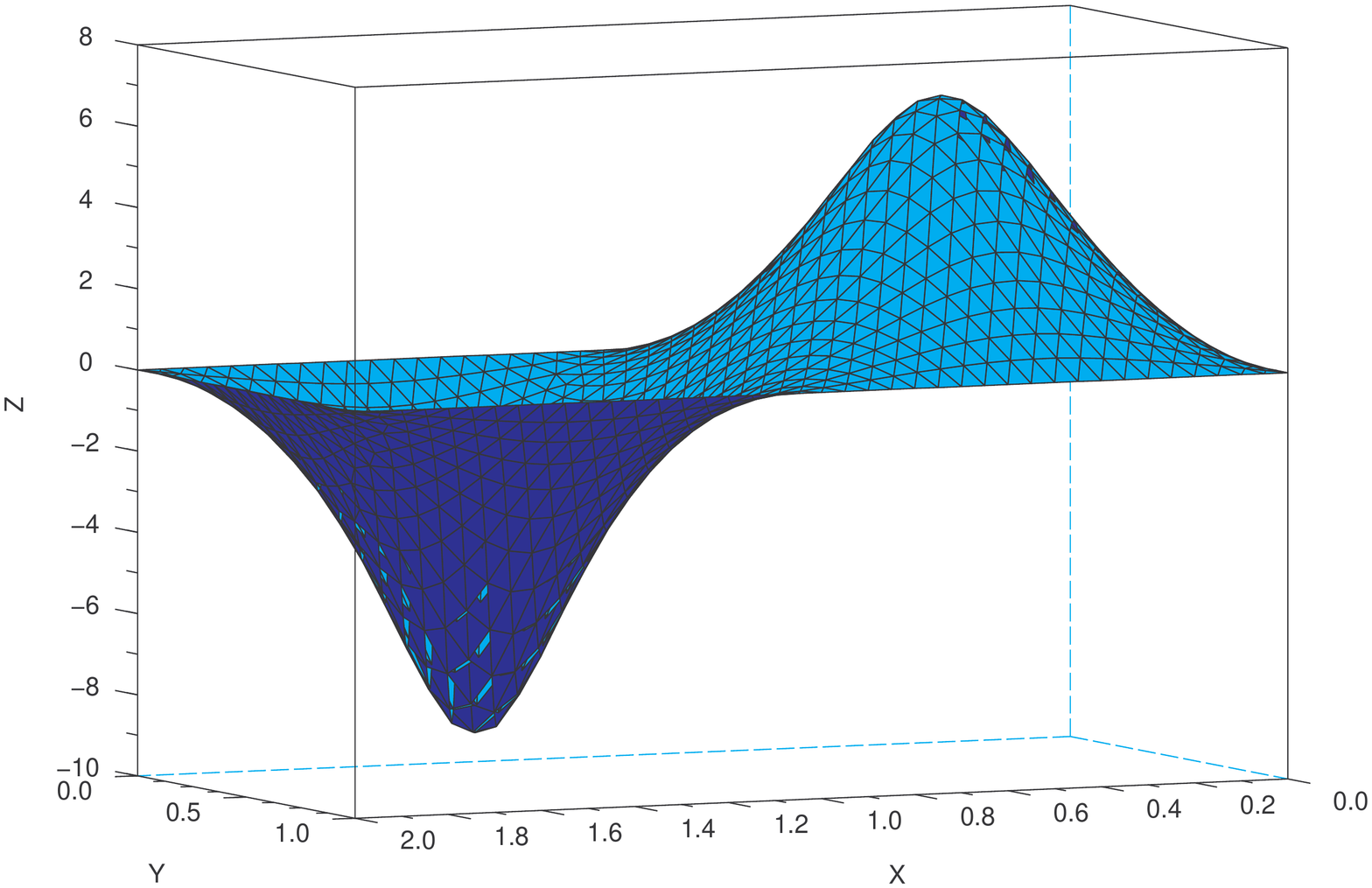}
\hfill\null
\begin{center}

\hfill\null
\end{center}
%\caption{Numerical solutions for constant $V$.} \label{V}
%\end{figure}%-------------------------------------------------------------------
\hfill\null
\end{minipage}
\hfill
\begin{minipage}[h]{0.4\linewidth}
\begin{itemize}
\item For g.s.: $\max (u)=5.98 $, $\mathcal{E}_4 (u)= 30.98$
\item For l.e.n.s.: $\min(u)= -8.67$, $\max (u)= 6.53$,  
$\mathcal{E}_4(u)=76.23$
\item Starting function for g.s.: $(x-2)(y-1)xy$
\item Starting function for l.e.n.s.: \\$\sin(\pi(x+1))\sin(2\pi (y+1))$
\end{itemize}
\end{minipage}
\hfill\null

If $V_-=0$ and $V_+=35$, we get the same symmetry  for g.s.\ but l.e.n.s.\
 are not symmetric. The mass is more or less ``located'' in the square
 defined by $x<1$. So, we obtain a
 direct symmetry breaking. To minimize the energy, the difference in
 the potential is so large that it is better to locate the mass in one
 side of the rectangle. On a
 square and for $V=0$, it is conjectured that l.e.n.s.\ is odd with respect to a
 diagonal. It explains the structure of the approximation. 
%\vspace{-0.1cm}

\begin{minipage}[h]{0.6\linewidth}
   % \begin{figure}
     \includegraphics[width=2.8cm, angle=270]{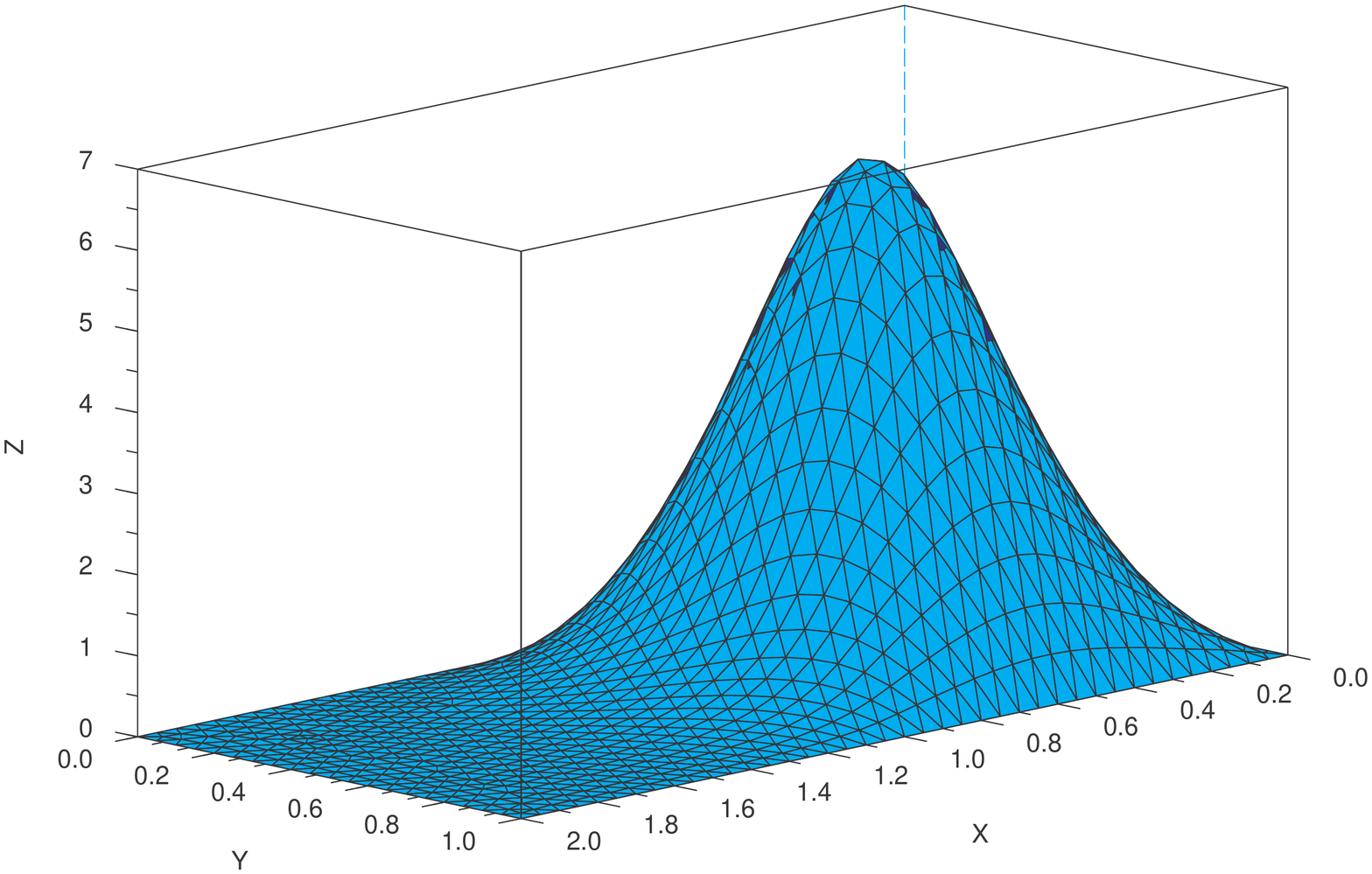}
\hfill
\includegraphics[width=2.8cm, angle = 270]{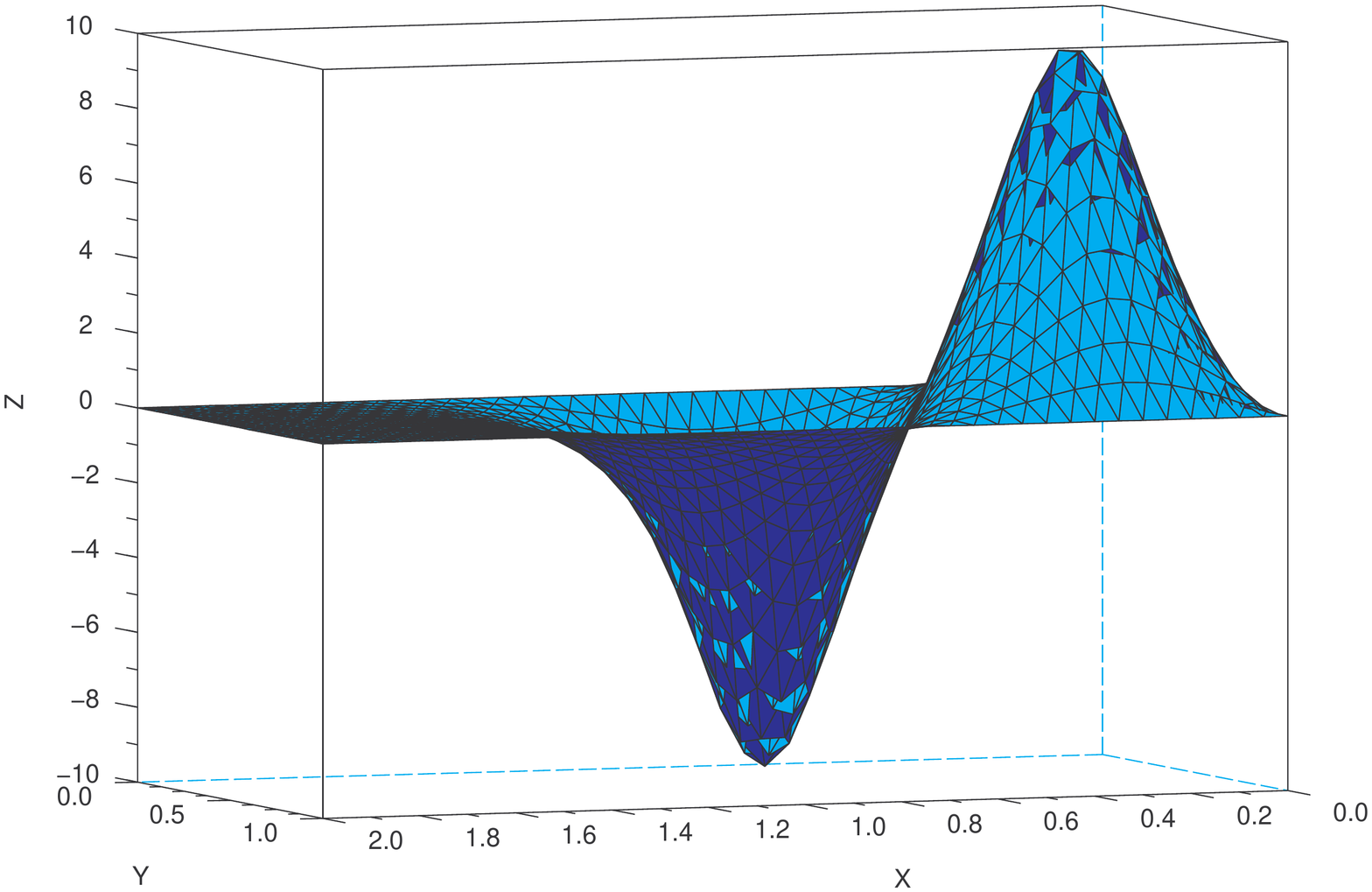}
\hfill\null
\begin{center}

\hfill\null
\end{center}
%\caption{Numerical solutions for constant $V$.} \label{V}
%\end{figure}%-------------------------------------------------------------------
\hfill\null
\end{minipage}
\hfill
\begin{minipage}[h]{0.4\linewidth}
\begin{itemize}
\item For g.s.: $\max (u)=6.19 $, $\mathcal{E}_4 (u)= 33.14$
\item For l.e.n.s.: $\min(u)= -9.8$, $\max (u)= 9.7$,  
$\mathcal{E}_4(u)=181.09$
\item Starting function for g.s.: $(x-2)(y-1)xy$
\item Starting function for l.e.n.s.: \\$\sin(\pi(x+1))\sin(2\pi (y+1))$
\end{itemize}
\end{minipage}
\hfill\null
%\vspace{-0.3cm}

\subsection{A singular potential on a ball}
%\vspace{-0.1cm}

As last example, we study singular potentials. First, $\lambda =1$, $p=4$ and 
$V(x,y)= \frac{1}{\sqrt{x^2+y^2}}$ on the ball $B(0,1)$.  Approximations show
 as expected that the ground state solutions are radial and the least energy
 nodal solutions are odd and even with respect to a diagonal. We
 obtain the same symmetry as for the potential $V=0$ (see~\cite{bbgv}). \\
\begin{minipage}[h]{0.6\linewidth}
   % \begin{figure}
     \includegraphics[width=2.8cm, angle=270]{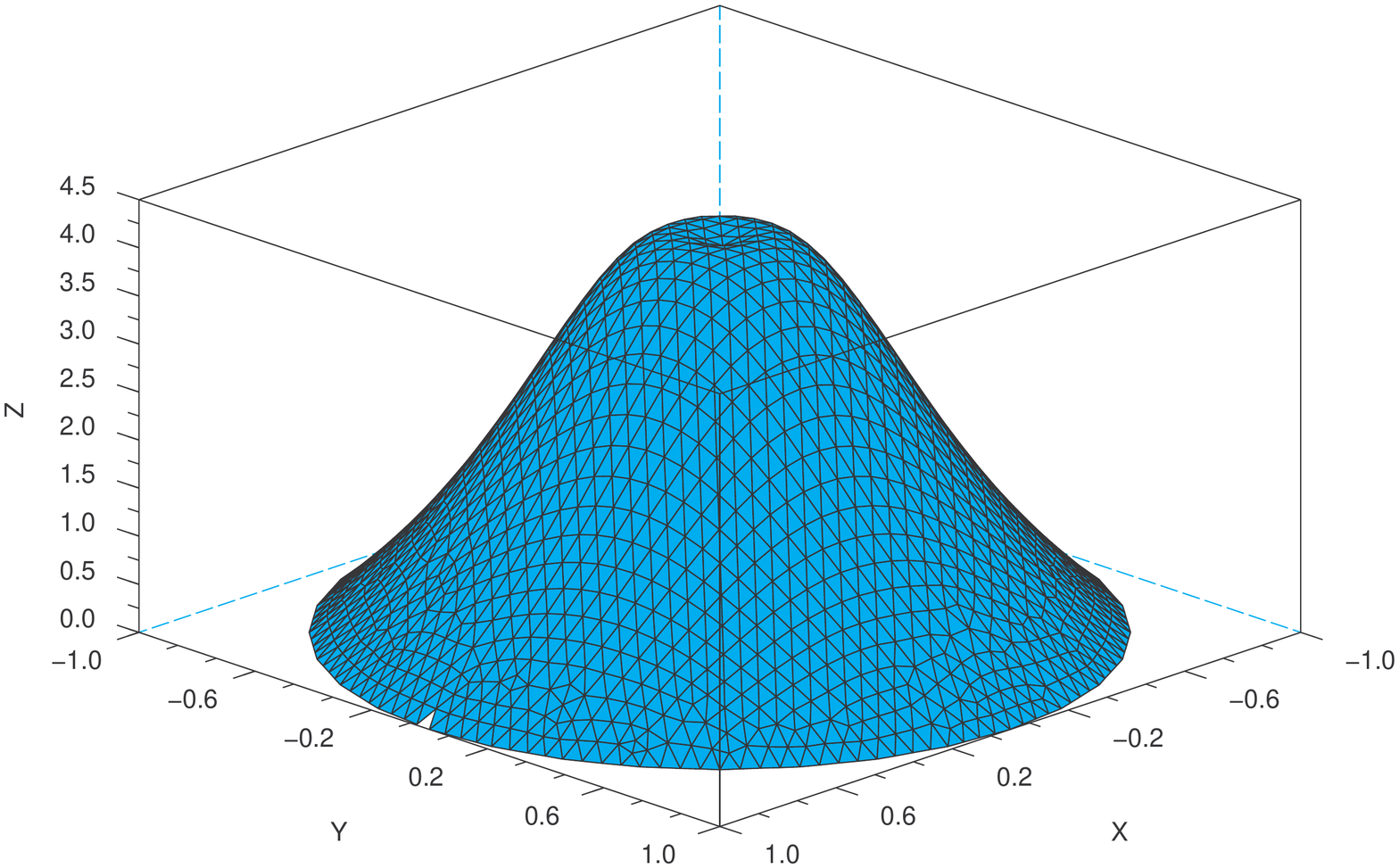}
\hfill
\includegraphics[width=2.8cm, angle = 270]{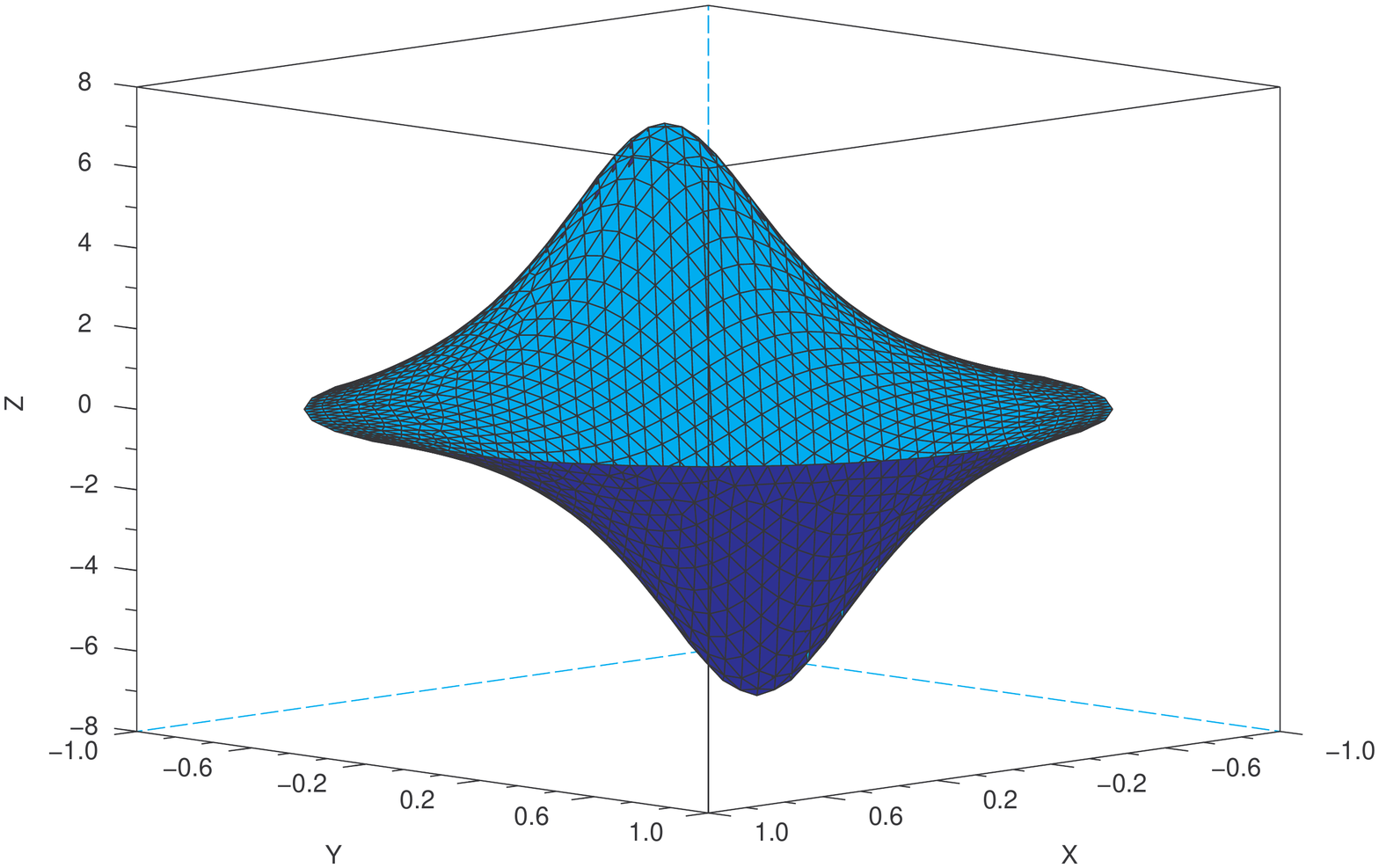}
\hfill\null
\begin{center}

\hfill\null
\end{center}
%\caption{Numerical solutions for constant $V$.} \label{V}
%\end{figure}%-------------------------------------------------------------------
\hfill\null
\end{minipage}
\hfill
\begin{minipage}[h]{0.5\linewidth}
\begin{itemize}
\item For g.s.: $\max (u)=4.15 $, $\mathcal{E}_4 (u)= 29.9$
\item For l.e.n.s.: $\min(u)= -6.36$, $\max (u)= 6.36$,  
$\mathcal{E}_4(u)=76.04$
\item Starting function for g.s.: \\$\cos(\pi (x^2+y^2)^{0.5}/2)$
\item Starting function for l.e.n.s.:\\ $\cos(\pi(x^2+y^2)^{0.5}/2)$ 
$\cos(2\pi(x^{2}+y^{2})^{0.5}$ $\cos(\pi (x^{2}+y^{2})^{0.5}$
\end{itemize}
\end{minipage}
\hfill\null

 Second, $V(x,y)= \frac{1}{\sqrt{(x-0.5)^2+y^2}}$ on the ball $B(0,1)$. 
Ground state solutions seem to be even with respect to a direction but are not radial. One can remark the
 work of the singularity on the level curve $1$. Least energy nodal
 solutions are just odd with respect  to  
a direction.  The mass is a little bit attracted by the side $x<0$ 
(the side minimizing the energy).\\

\begin{minipage}[h]{0.6\linewidth}
   % \begin{figure}
     \includegraphics[width=2.8cm, angle=270]{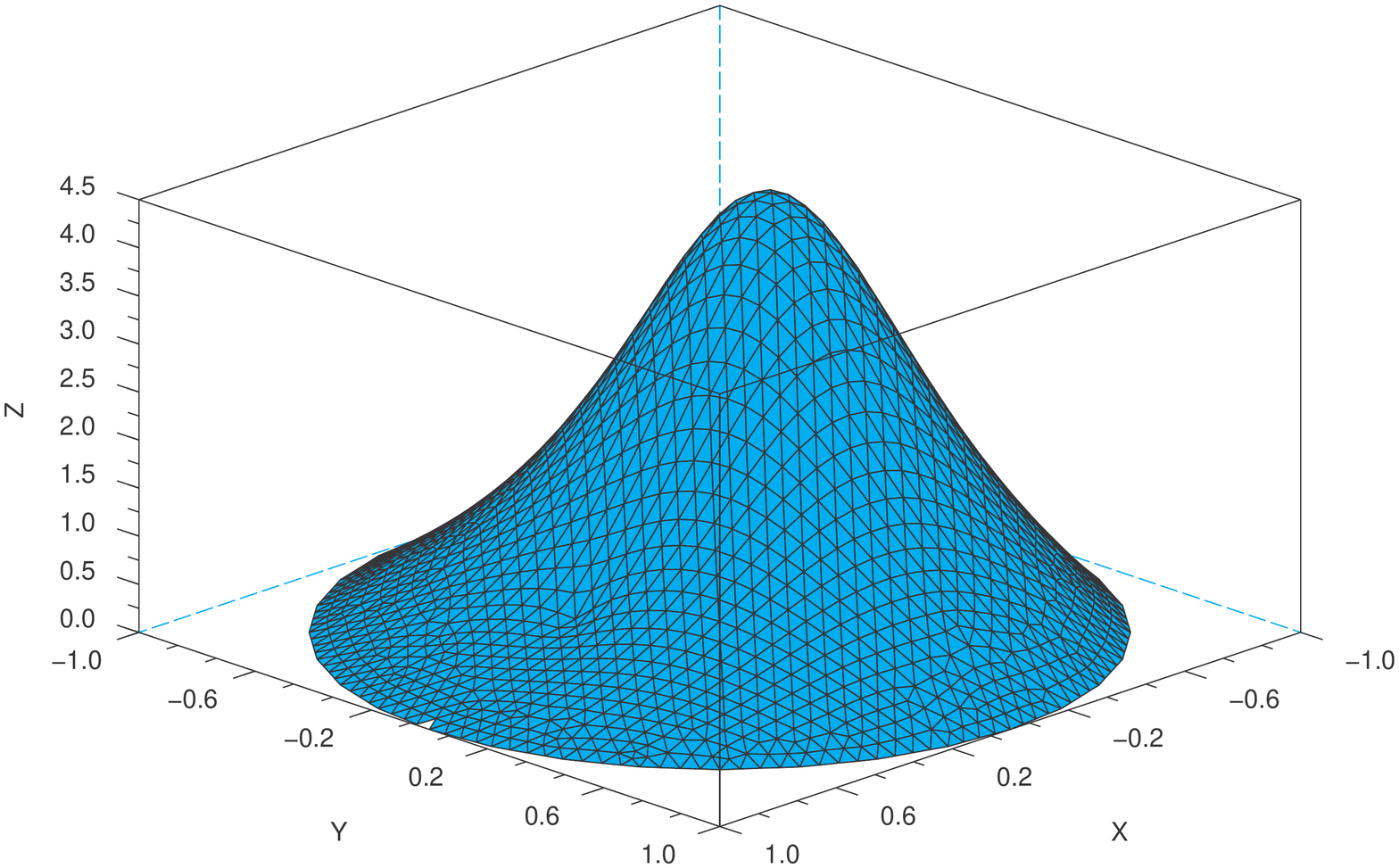}
\hfill
\includegraphics[width=2.8cm, angle = 270]{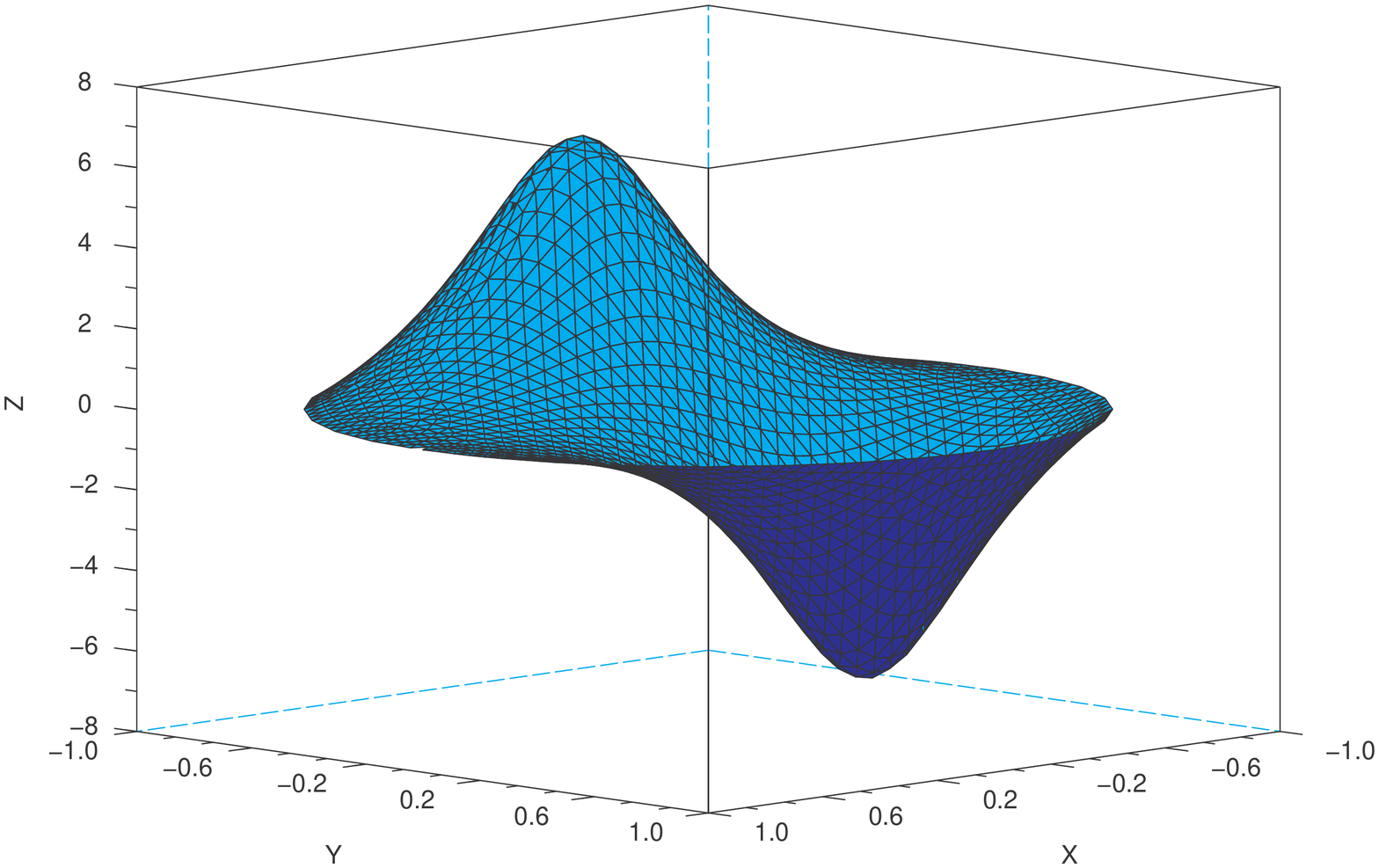}
\hfill\null
\begin{center}

\hfill\null
\end{center}
%\caption{Numerical solutions for constant $V$.} \label{V}
%\end{figure}%-------------------------------------------------------------------
\hfill\null
\end{minipage}
\hfill
\begin{minipage}[h]{0.5\linewidth}
\begin{itemize}
\item For g.s.: $\max (u)=4.41 $, $\mathcal{E}_4 (u)= 18.74$
\item For l.e.n.s.: $\min(u)= -6.25$, $\max (u)= 6.25$,  
$\mathcal{E}_4(u)=76.23$
\item Starting function for g.s.: \\$\cos(\pi (x^2+y^2)^{0.5}/2)$
\item Starting function for l.e.n.s.: \\
$\cos(\pi (x^{2}+y^{2})^{0.5}/2)$ $ \cos(2\pi
(x^{2}+y^{2})^{0.5}$ $\cos(\pi (x^{2}+y^{2})^{0.5}$
\end{itemize}
\end{minipage}
\hfill\null

%\vspace{-1cm}
\bibliographystyle{plain}
\bibliography{these}
\end{document}